\documentclass[11pt,notitlepage,a4paper,reqno]{amsart}
\usepackage{graphicx}
\usepackage{amsmath,amsfonts,amssymb,amsthm,epsfig,,mathrsfs,color,amsfonts}
\usepackage{bbm}

\newtheorem{theorem}{Theorem}
\newtheorem{corollary}{Corollary}

\newtheorem{remark}[corollary]{Remark}
\newtheorem{definition}[corollary]{Definition}

\newcommand{\R}{\mathbb{R}}

\numberwithin{equation}{section}

\def\H{{\mbox{\script H}\;}}
\def\N{\mathbb N}
\def\R{\mathbb R}
\def\H{\mathcal H}
\def\e{\varepsilon}

\def\O{\Omega}

\def\ov{\overline}
\def\S{\Sigma}

\def\bal{\begin{aligned}}
\def\eal{\end{aligned}}
\newcommand{\Per}{{\rm Per}}
\newcommand\LL{\mathcal{L}}
\newcommand\EE{\mathscr{E}}
\newcommand{\MS}{{\rm MS}}

\title[Hausdorff dimension of the singular set of Mumford-Shah 
minimizers]{A note on the Hausdorff dimension of the singular set 
for minimizers of the Mumford-Shah energy}

\author{C. De Lellis}
\address{Universit\"at Z\"urich, Winterthurerstrasse 190, CH-8057 Z\"urich}
\email{camillo.delellis@math.uzh.ch}

\author{M. Focardi}
\address{Universit\`a di Firenze, DiMaI ``U. Dini'', V.le Morgagni 67/A, 
I-50134 Firenze}
\email{focardi@math.unifi.it}

\author{B. Ruffini}
\address{Scuola Normale Superiore di Pisa, Piazza dei Cavalieri 4, I-56126 Pisa}
\email{berardo.ruffini@sns.it}

\begin{document}

\begin{abstract}
We give a more elementary proof of a result by Ambrosio, Fusco and Hutchinson to estimate 
the Hausdorff dimension of the singular set of minimizers of the Mumford-Shah energy
(see \cite[Theorem~5.6]{AFH}). On the one hand, we follow the strategy of the above mentioned paper; 
but on the other hand our analysis greatly simplifies the argument since it relies 
on the compactness result proved by the first two Authors in \cite[Theorem~13]{DeLF} for sequences of 
local minimizers with vanishing gradient energy, and the regularity theory of minimal 
Caccioppoli partitions, rather than on the corresponding results for Almgren's area 
minimizing sets.
\end{abstract}

\maketitle

\section{Introduction}
Consider the (localized) Mumford-Shah energy on a bounded open subset $\O\subset\R^n$ 
given by  
\begin{equation}\label{e:MSenergy}
\MS(v,A)=\int_A|\nabla v|^2dx+\H^{n-1}(S_v\cap A),\qquad
\text{for } v\in SBV(\O) \text{ and } A\subseteq\O \text{ open.} 
\end{equation}
In what follows if $A=\Omega$ we shall drop the dependence on the set of integration. 
We refer to the book \cite{AFP00} for all the notations and preliminaries on $SBV$ functions 
and the regularity theory for local minimizers of the Mumford-Shah energy giving precise 
references when needed.

In this note we provide a simplified proof of the following result due to Ambrosio, Fusco 
and Hutchinson \cite[Theorem 5.6]{AFH} (established there for quasi-minimizers as well).
\begin{theorem}\label{t:main}
Let $u$ be a local minimizer of the Mumford-Shah energy, i.e. 
any function $u\in SBV(\O)$ with $\MS(u)<\infty$ and such that
\[
\MS(u)\leq \MS(w)\quad \text{ whenever } \{w\neq u\}\subset\subset\O.
\]
Let $\Sigma_u\subseteq \ov{S_u}$ be the set of points out of 
which $\ov{S_u}$ is locally regular, and let
\[
\Sigma^\prime_u:=\left\{x\in\Sigma_u:\, 
\lim_{\rho\downarrow 0}\rho^{1-n}\int_{B_\rho(x)}|\nabla u|^2=0\right\}.
\]
Then, $\mathrm{dim}_{\H}\S^\prime_u\le n-2$. 

\end{theorem}
The main interest in establishing such an estimate on the set $\S^\prime_u$, the so called subset of triple-junctions, 
is related to the understanding of the Mumford-Shah conjecture (see \cite[Chapter~6]{AFP00} for a related discussion, 
see also \cite[Section~7]{DeLF}). 

Indeed, Theorem~\ref{t:main}, together with the higher integrability property of the approximate gradients enjoyed by 
minimizers as established in $2$-dimensions by \cite{DeLF} and more recently in any dimension by \cite{DePF}, imply 
straightforwardly an analogous estimate on the full singular set $\Sigma_u$. 
More precisely, in view of \cite[Theorem~1]{DeLF} and \cite[Theorem~1.1]{DePF} any local minimizer $u$ of the Mumford-Shah 
energy is such that $|\nabla u|\in L^p_{\mathrm{loc}}(\O)$ for some $p>2$, therefore \cite[Corollary~5.7]{AFH} yields that
\[
\mathrm{dim}_{\H}\S_u\le \max\{n-2,n-p/2\}.
\]
A characterization (of a suitable version) of the Mumford-Shah conjecture in $2$-dimensions in terms of a refined higher 
integrability property of the gradient in the finer scale of weak Lebesgue spaces has been recently established in 
\cite[Proposition~5]{DeLF}.

Our proof of Theorem~\ref{t:main} rests on a compactness result proved 
by  the first two Authors (see \cite[Theorem~13]{DeLF}) 
showing that the blow-up limits of the jump set $S_u$ in points in the 
regime of small gradients, i.e. in points of $\Sigma^\prime_u$, are minimal 
Caccioppoli partitions.
The original approach in \cite{AFH}, instead, relies on the notion of
Almgren's area mimizing sets, for which an interesting but technically 
demanding analysis of the composition of $SBV$ functions with Lipschitz 
deformations (not necessarily one-to-one) and a revision of the regularity 
theory for those sets are needed (cp. with \cite[Sections~2, 3 and 4]{AFH}).

Given \cite[Theorem~13]{DeLF}, the regularity theory of minimal Caccioppoli 
partitions developed in \cite{MT,L1,L2} and standard arguments in geometric measure 
theory yield the conclusion, thus bypassing the above mentioned technical complications.

We describe briefly the plan of the note: in Section~\ref{s:cacc} we introduce 
necessary definitions and recall some well-known facts about Caccioppoli partitions.
In Section~\ref{s:stimasing} we prove our main result and comment on some related improvements 
in a final remark.

\section{Caccioppoli partitions}\label{s:cacc}

\noindent In what follows $\O\subset\R^n$ will denote a bounded open set. 
\begin{definition}\label{d:CP}
A Caccioppoli partition of $\Omega$ is a countable partition
$\EE = \{E_i\}_{i=1}^\infty$ of $\Omega$ in sets of 
(positive Lebesgue measure and) finite perimeter with
$\sum_{i=1}^\infty \Per (E_i, \Omega) < \infty$. 

For each Caccioppoli partition $\EE$ we define its set of interfaces as
\[
J_{\EE}:= \bigcup_{i\in\N} \partial^* E_i\, .
\]
The partition $\EE$ is said to be minimal if 
\[
\H^{n-1}(J_{\mathscr E})\leq \H^{n-1}(J_{\mathscr F})
\]
for all Caccioppoli partitions ${\mathscr F}$ for which there exists
an open subset $\Omega'\subset\subset \Omega$ with 
\[
\sum_{i=1}^\infty\LL^n\left((F_i\triangle E_i)
\cap(\Omega\setminus \Omega')\right)=0.
\]
\end{definition}

\begin{definition}\label{d:Cacc_sing_set}
Given a Caccioppoli partition $\EE$ we define its singular set
$\Sigma_\EE$ as the set of points for which the approximate tangent
plane to $J_\EE$ does not exist.  
\end{definition}
A characterization of the singular set $\Sigma_\EE$ for minimal Caccioppoli 
partitions in the spirit of $\e$-regularity results is provided in the ensuing 
statement (cp. with \cite[Corollary 4.2.4]{L2} and \cite[Theorem III.6.5]{M} ). 
\begin{theorem}\label{t:Cacc_sing_est}
Let $\O$ be an open set and $\EE=\{E_i\}_{i\in\N}$ a minimal Caccioppoli 
partition of $\O$. 

Then, there exists a dimensional constant $\e=\e(n)>0$ such that 
\begin{equation}\label{e:caratterizzazione}
\Sigma_\EE=\left\{x\in \O\cap\ov{J_\EE}:\,
\inf_{B_\rho(x)\subset\subset\O}e(x,\rho)\ge \e  \right\},
\end{equation}
where $e(x,\rho)$ denotes the spherical excess of $\EE$ at the 
point $x\in J_\EE$ at the scale $\rho>0$, that is
\[
e(x,\rho):=\min_{\nu\in \mathbb S^{N-1}}\frac{1}{\rho^{n-1}}
\int_{B_\rho(x)\cap J_\EE}\frac{|\nu_\EE(y)-\nu|^2}{2}d\H^{n-1}(y).
\]
\end{theorem}

We recall next a result that is probably well-known in literature; 
we provide the proof for the sake of completeness. 
\begin{theorem}\label{t:Caccest}
Let $\EE$ be a minimal Caccioppoli partition in $\O$, 
then $\mathrm{dim}_{\H}\Sigma_\EE\le n-2$.

If, in addition, $n=2$, then $\Sigma_\EE$ is locally finite.
\end{theorem}
\begin{proof}
We apply  the abstract version of Federer's reduction argument in 
\cite[Theorem A.4]{S} with the set of functions
\[
\mathcal F=\{\chi_{J_\EE}:\, \EE \text{ is a minimal Caccioppoli partition}\}
\]
endowed with the convergence 
\[
\chi_{J_{\EE_h}}\to\chi_{J_\EE}\,\Longleftrightarrow\,
\lim_{h\uparrow\infty}\int_{J_{\EE_h}}g\,d\H^{n-1}=\int_{J_\EE}g\,d\H^{n-1},\quad
\text{for all }g\in C^1_c(\O).
\]
and singularity map $\mathrm{sing}(\chi_\EE)=\Sigma_\EE$. 

It is easy to see that condition $A.1$ (closure under scaling) 
and $A.3(2)$ hold true. Moreover, the blow-ups of a minimal Caccioppoli 
partition converge to a minimizing cone (see \cite[Theorem 3.5]{L1}, or 
\cite[Theorem 4.4.5 (a)]{L2}), so that $A.2$ holds as well. 
About $A.3(1)$, we notice that the singular set of an hyperplane is empty. 
Eventually, if a sequence $(\chi_{J_{\EE_h}})_{h\in\N}\subseteq \mathcal F$ 
converges to $\chi_{J_\EE}$ and $(x_h)_{h\in\N}$ converges to $x$, with 
$x_h\in \S_{\EE_h}$ for all $h$, then by the continuity of the excess and
the characterization in \eqref{e:caratterizzazione}, $x\in\Sigma_\EE$, so 
that condition $A.3(3)$ is satisfied as well.

To conclude, we recall that \cite[Theorem A.4]{S} itself ensures 
that the set $\S_\EE$ is locally finite being in this setting 
$\mathrm{dim}_{\H}\Sigma_\EE=0$. 
\end{proof}

\section{Proof of the main result}\label{s:stimasing}

We are now ready to prove the main result of the note following the approach 
exploited in \cite[Theorem~5.6]{AFH}. To this aim we recall that Ambrosio, 
Fusco \& Pallara (see \cite[Theorems 8.1-8.3]{AFP00}) characterized alternatively 
the singular set $\S_u$ as follows
\begin{equation}\label{e:sigmau}
\Sigma_u=\{x\in\overline{S_u}:\,\liminf_{\rho\downarrow 0}
\left(\mathscr{D}(x,\rho)+\mathscr{A}(x,\rho)\right)\ge\e_0\},
\end{equation}
where $\e_0$ is a dimensional constant, and the scaled Dirichlet 
energy and the scaled mean-flatness are respectively defined as
\[
\mathscr{D}(x,\rho):=\rho^{1-n}\int_{B_\rho(x)}|\nabla u|^2dy,\quad
\mathscr{A}(x,\rho):=\rho^{-1-n}\min_{T\in\Pi }\int_{S_u\cap B_\rho(x)}
\mathrm{dist}^2(y,T)d\H^{n-1}(y),
\]
with $\Pi$ the set of all affine $(n-1)$-planes in $\R^n$.

\begin{proof}[Proof of Theorem~\ref{t:main}]
We argue by contradiction: suppose that there exists $s> n-2$ such that 
$\H^s(\Sigma^\prime_u)>0$. 
From this, we infer that $\H^s_\infty(\Sigma^\prime_u)>0$, 
and moreover that for $\H^s$-a.e. $x\in \Sigma^\prime_u$ it holds
\begin{equation}\label{e:limsup}
\limsup_{\rho\downarrow 0^+}
\frac{\H^s_\infty(\Sigma^\prime_u\cap B_\rho(x))}{\rho^s}
\ge \frac{\omega_s}{2^s}
\end{equation}
(see for instance \cite[Theorem 2.56 and formula (2.43)]{AFP00} or 
\cite[Lemma III.8.15]{M}). 
Without loss of generality, suppose that \eqref{e:limsup} holds at $x=0$, and 
consider a sequence $\rho_h\downarrow 0$ for which
\begin{equation}\label{e:assurdo}
\H^s_\infty(\Sigma^\prime_u\cap B_{\rho_h})\ge\frac{\omega_s}{2^{s+1}}\rho_h^s
\qquad\text{for all }h\in\N.
\end{equation}
\cite[Theorem~13]{DeLF} provides a subsequence, not relabeled for 
convenience, and a minimal Caccioppoli partition $\EE$ such that 
\begin{equation}\label{e:DeLF}
\H^{N-1}\llcorner\,\rho_h^{-1}S_u
\stackrel{\ast}{\rightharpoonup}\H^{N-1}\llcorner\, J_\EE,\,
\text{ and }\quad\rho_h^{-1}\ov{S_u}\to\ov{J_\EE}\,\,\text{ locally Hausdorff}.
\end{equation}
In turn, from the latter we claim that if $\mathcal{F}$ is any open cover  
of $\Sigma_\EE\cap \ov B_1$, then for some $h_0\in\N$ 
\begin{equation}\label{e:usc}
\rho_h^{-1}\Sigma^\prime_u\cap\ov B_1\subseteq 
\cup_{F\in\mathcal{F}}F\qquad \text{for all } h\ge h_0. 
\end{equation}
Indeed, if this is not the case we can find a sequence 
$x_{h_j}\in \rho_{h_j}^{-1}\Sigma^\prime_u\cap\ov B_1$ 
converging to some point $x_0\notin\S_\EE$. 
If $\pi^\EE_{x_0}$ is the approximate tangent plane to $J_\EE$ at $x_0$, 
that exists by the very definition of $\Sigma_\EE$, then for some $\rho_0$ we have
\[
\rho^{1-n}\int_{B_\rho(x_0)\cap J_\EE}
\mathrm{dist}^2(y,\pi^\EE_{x_0})d\H^{n-1}<\e_0,
\quad \text{ for all }\rho\in(0,\rho_0).
\]
In turn, from the latter inequality it follows for $\rho\in(0,\rho_0\wedge 1)$
\[
\limsup_{j\uparrow\infty}\int_{B_\rho(x_{h_j})\cap\rho_{h_j}^{-1}S_u}
\mathrm{dist}^2(y,\pi^\EE_{x_0})\,d\H^{n-1}<\e_0.
\]
Therefore, as $x_{h_j}\in\rho_{h_j}^{-1}\Sigma^\prime_u$, we get for 
$j$ large enough 
\[
\limsup_{\rho\downarrow 0}
\left(\mathscr{D}(x_{h_j},\rho)+\mathscr{A}(x_{h_j},\rho)\right)<\e_0,
\]
a contradiction in view of the characterization of the singular set in \eqref{e:sigmau}. 

To conclude, we note that by \eqref{e:usc} we get
\[
\H^s_\infty(\S_\EE\cap\ov{B_1})\geq
\limsup_{h\uparrow\infty}\H^s_\infty(\rho_h^{-1}\Sigma^\prime_u\cap\ov{B_1});
\]
given this, \eqref{e:assurdo} and \eqref{e:DeLF} yield that
\[
\H^s(\Sigma_\EE\cap \ov{B_1})\ge\H^s_\infty(\Sigma_\EE\cap\ov{B_1})\ge
\limsup_{h\uparrow\infty}\H^s_\infty(\rho_h^{-1}\Sigma^\prime_u\cap\ov{B_1})\ge
\frac{\omega_s}{2^{s+1}},
\]
thus contradicting Theorem~\ref{t:Caccest}.
\end{proof}

\begin{remark}
In two dimensions we can actually prove that the set  
$\Sigma^\prime_u$ of triple-junctions is at most countable building 
upon some topological arguments.
This claim follows straightforwardly from the compactness result \cite[Theorem~13]{DeLF}, 
David's $\e$-regularity theorem \cite[Proposition~60.1]{David}, and 
a direct application of Moore's triod theorem showing that in the plane 
every system of disjoint triods, i.e. unions of three Jordan arcs that 
have all one endpoint in common and otherwise disjoint, is at most countable
(see \cite[Theorem~1]{Moore} and \cite[Proposition~2.18]{pommerenke}). 
Despite this, we are not able to infer that $\S^\prime_u$ is locally finite as 
in the case of minimal Caccioppoli partitions (cp. with Theorem~\ref{t:Caccest}). 
Indeed, if on one hand we can conclude that every convergent sequence 
$(x_j)_{j\in\N}\subset\S^\prime_u$ has a limit $x_0\notin\S^\prime_u$ thanks to 
\cite[Proposition~11 and Lemma~12]{DeLF}; on the other hand, we cannot exclude 
that the limit point $x_0$ is a crack-tip, i.e. it belongs to the set 
$\S_u\setminus\S^\prime_u=\{x\in\Sigma_u:\,\liminf_{\rho\downarrow 0}\mathscr{D}(x,\rho)>0\}$.

The same considerations above apply in three dimensions as well for points whose 
blow-up is a $\mathbb{T}$ cone, i.e. a cone with vertex the origin constructed upon the 1-skeleton
of a regular tetrahedron. The latter claim follows thanks to \cite[Theorem~13]{DeLF}, 
the $3$-d extension of David's $\e$-regularity result by Lemenant in \cite[Theorem~8]{lemenant}, 
and a suitable extension of Moore's theorem on triods established by Young in \cite{Young}.

Let us finally point out that we employ topological arguments to compensate for monotonicity
formulas, that would allow us to exploit Almgren's stratification type results and get, actually, 
a more precise picture of the set $\Sigma^\prime_u$ (cp.~with \cite[Theorem 3.2]{White}).
\end{remark}



\begin{thebibliography}{100}
%
\bibitem{AFP00} L. Ambrosio, N. Fusco \& D. Pallara.
Functions of bounded variation and free discontinuity problems,
in the Oxford Mathematical Monographs.
The Clarendon Press Oxford University Press, New York, 2000.
%
\bibitem{AFH} L. Ambrosio, N. Fusco \& J.E. Hutchinson. 
Higher integrability of the gradient and dimension of the singular 
set for minimisers of the Mumford-Shah functionals, 
Calc. Var. Partial Differential Equation {\bf 16} (2003), no. 2, 187--215.
%
\bibitem{David} G. David. 
Singular Sets of Minimizers for the Mumford-Shah Functional, 
Progress in Mathematics, vol. {\bf 233}, Birkh\"auser Verlag, Basel, 2005.
%
\bibitem{DeLF} C. De Lellis \& M. Focardi. 
Higher integrability of the gradient for minimizers of the $2d$ 
Mumford-Shah energy, in press on J. Math. Pures Appl., 
http://dx.doi.org/10.1016/j.matpur.2013.01.006 
%
\bibitem{DePF} G. De Philippis \& A. Figalli. 
Higher integrability for minimizers of the Mumford-Shah functional, preprint 2013, http://cvgmt.sns.it/paper/2111/
%
\bibitem{lemenant} A. Lemenant. 
Regularity of the singular set for Mumford-Shah minimizers in $\Bbb R^3$ near a minimal cone. 
Ann. Sc. Norm. Super. Pisa Cl. Sci. (5)  {\bf 10}  (2011), no. 3, 561--609.
%
\bibitem{L1} G.P. Leonardi. 
Blow-Up of Oriented Boundaries,  
Rend. Sem. Mat. Univ. Padova  {\bf 103} (2000), 211--232. 
%
\bibitem{L2} G.P. Leonardi. 
Optimal Subdivisions of $n$-dimensional Domains, 
PhD thesis, Universit\`a di Trento 1998.
%
\bibitem{M} F. Maggi. 
Sets of Finite Perimeter and Geometric Variational Problems,
Cambridge Studies in Advanced Mathematics {\bf 135}, Cambridge 
University Press, 2012. 
%
\bibitem{MT} U. Massari \& I. Tamanini. 
Regularity properties of optimal segmentations, 
J. Reine Angew. Math. {\bf 420} (1991), 61--84.  
%
\bibitem{Moore} R.L. Moore.
Concerning triods in the plane and the junction points of plane continua.
Proceedings of the National Academy of Sciences of the United States of America,
Vol. {\bf 14}, No. 1 (Jan. 15, 1928), 85--88.
%
\bibitem{pommerenke} C. Pommerenke.
Boundary behaviour of conformal maps.
Grundlehren der Mathematischen Wissenschaften 
[Fundamental Principles of Mathematical Sciences], {\bf 299}. 
Springer-Verlag, Berlin, 1992.
%
\bibitem{S} L. Simon.
Lectures on geometric measure theory. 
Proceedings of the Centre for Mathematical Analysis, 
Australian National University, 3. Australian National University, 
Centre for Mathematical Analysis, Canberra, 1983. vii+272.
%
\bibitem{White} B. White, 
Stratification of minimal surfaces, mean curvature flows, and harmonic maps.
J. reine angew. Math. {\bf 488} (1997), 1--35.
%
\bibitem{Young} G.S. Young, Jr. 
A generalization of Moore's theorem on simple triods.  
Bull. Amer. Math. Soc.  {\bf 50}, (1944). 714.
%
\end{thebibliography}
\end{document}